\documentclass[12pt]{article}
\usepackage{amsfonts, amsmath, amssymb, amsthm, mathtools, dsfont}
\usepackage[colorlinks=true,citecolor=blue]{hyperref}
\usepackage[margin=3cm]{geometry}
\usepackage[utopia]{mathdesign}
\usepackage[mathscr]{euscript}
\usepackage[utf8]{inputenc}

\usepackage{graphicx}
\usepackage{parskip}
\usepackage{enumerate}
\usepackage{verbatim}
\usepackage{tikz}
\usepackage{units}

\usetikzlibrary{calc}
\usepackage[outline]{contour}
\contourlength{1.2pt}

\begingroup
    \makeatletter
    \@for\theoremstyle:=definition,remark,plain\do{%
        \expandafter\g@addto@macro\csname th@\theoremstyle\endcsname{%
            \addtolength\thm@preskip\parskip
            }%
        }
\endgroup

\newtheorem{theorem}{Theorem}[section]
\newtheorem*{theorem*}{Theorem}
\newtheorem{lemma}[theorem]{Lemma}
\newtheorem{prop}[theorem]{Proposition}

\newtheorem{prob}[theorem]{Problem}
\newtheorem{conj}[theorem]{Conjecture}
\newtheorem{que}[theorem]{Question}

\theoremstyle{definition}
\newtheorem{definition}[theorem]{Definition}

\newtheorem*{remark*}{Remark}
\newtheorem{claim}[theorem]{Claim}

\newcommand{\df}[1]{{{\color{blue!50!black}\em #1}}}
\DeclareMathOperator{\conv}{conv}
\DeclareMathOperator{\intr}{int}
\DeclareMathOperator{\Vol}{Vol}

\newcommand{\cceil}[1]{\left\lceil {#1} \right\rceil}
\newcommand{\ffloor}[1]{{\left\lfloor {#1} \right\rfloor}}

\begin{document} 

\title{The layer number of $\alpha$-evenly distributed point sets}
\author{
	Ilkyoo Choi\thanks{
Supported by the Basic Science Research Program through the National Research Foundation of Korea (NRF) funded by the Ministry of Education (NRF-2018R1D1A1B07043049), and also by the Hankuk University of Foreign Studies Research Fund.
Department of Mathematics, Hankuk University of Foreign Studies, Yongin-si, Gyeonggi-do, Republic of Korea.
\texttt{ilkyoo@hufs.ac.kr}
	}
	\and
	Weonyoung Joo\thanks{
		Department of Industrial and Systems Engineeing, KAIST, Daejeon, Republic of Korea.
		\texttt{es345@kaist.ac.kr}
	}
	\and
	Minki Kim\thanks{
		Partially supported by ISF grant no. 2023464 and BSF grant no. 2006099.
		Department of Mathematics, Technion -- Israel Institute of Technology, Haifa, Israel.
		\texttt{kimminki@technion.ac.il}
	}
}
\date\today
\maketitle

\begin{abstract}
	For a finite point set in $\mathbb{R}^d$, we consider a peeling process where the vertices of the convex hull are removed at each step.
	The layer number $L(X)$ of a given point set $X$ is defined as the number of steps of the peeling process in order to delete all points in $X$.
	It is known that if $X$ is a set of random points in $\mathbb{R}^d$, then the expectation of $L(X)$ is $\Theta(|X|^{2/(d+1)})$, and recently it was shown that if $X$ is a point set of the square grid on the plane, then $L(X)=\Theta(|X|^{2/3})$.
	
	In this paper, we investigate the layer number of $\alpha$-evenly distributed point sets for $\alpha>1$; these point sets share the regularity aspect of random point sets but in a more general setting. 
The set of lattice points is also an $\alpha$-evenly distributed point set for some $\alpha>1$. 
We find an upper bound of $O(|X|^{3/4})$ for the layer number of an $\alpha$-evenly distributed point set $X$ in a unit disk on the plane for some $\alpha>1$, and provide an explicit construction that shows the growth rate of this upper bound cannot be improved.
	In addition, we give an upper bound of $O(|X|^{\frac{d+1}{2d}})$ for the layer number of an $\alpha$-evenly distributed point set $X$ in a unit ball in $\mathbb{R}^d$ for some $\alpha>1$ and $d\geq 3$.
\end{abstract}


\section{Introduction}\label{sec:intro}

Let $\Vol(R)$ be the volume function of a region $R$ in $\mathbb{R}^d$.
	For a finite point set $X$ in $\mathbb{R}^d$, let $\conv(X)$ denote the convex hull of $X$ and let $V(X)$ denote the \df{convex layer} of $X$, which is the set of extreme points of $X$.
	We consider a ``peeling process'' of $X$ that removes the convex layer at each step.
Given a point set $X$, the \df{peeling sequence} $\{X_i\}_i$ of $X$ is the sequence of subsets of $X$ such that $X_i$ is the set of remaining vertices at step $i-1$;
	in other words, $X_1 =X$ and  $X_{i+1} = X_i\setminus V(X_i)$ for $i\geq1$.
See Figure~\ref{fig:peeling} for an illustration.

\begin{figure}[h]
    \centerline{\includegraphics[scale=0.9]{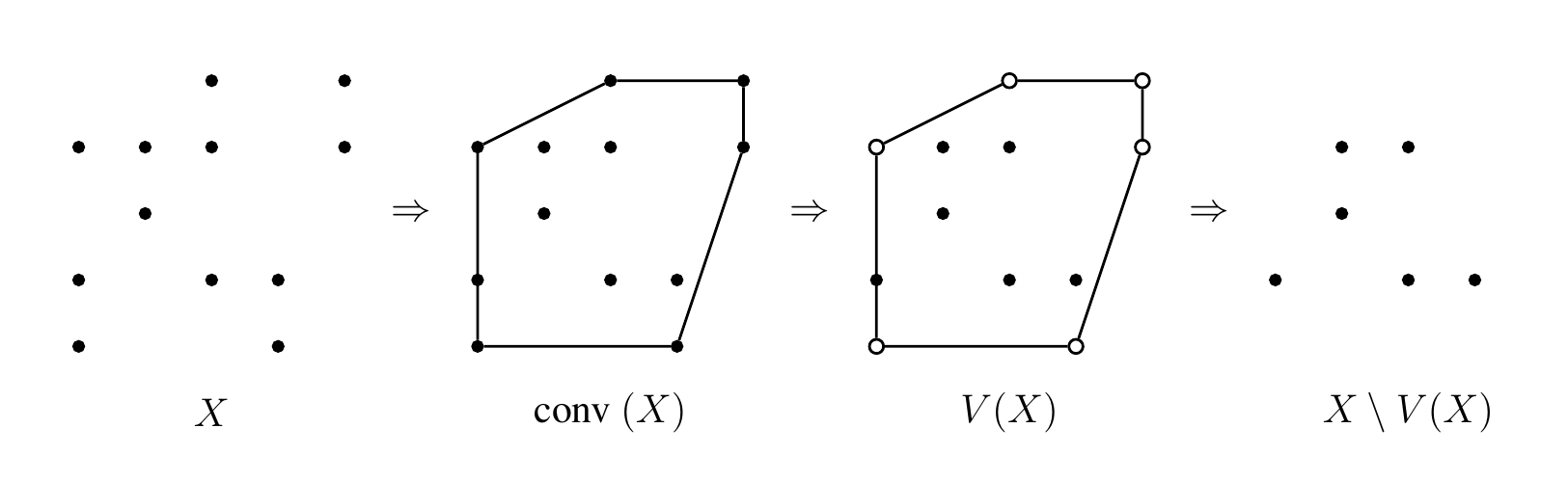}}
    \caption{One step of the peeling process of $X$.}
    \label{fig:peeling}
\end{figure}

	The \df{layer number} $L(X)$ of $X$ is the number of steps the peeling process takes to remove all points from $X$; namely, $L(X):=\min\{k:X_{k+1}=\emptyset\}$.
	An optimal deterministic algorithm to find $L(X)$ with running time $O(|X| \log (|X|))$ was discovered in~\cite{Cha85}. 

Note that $L(X)=1$ if and only if $X$ is a point set in convex position.
It is not hard to see that $1\leq L(X)\leq \lceil|X|/2\rceil$.
The upper bound is sharp, as it comes from points on a single straight line.
See the illustration in Figure~\ref{fig:ex}.

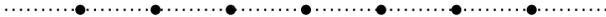
\begin{figure}[htbp]
\centering
\begin{tikzpicture}
\tikzstyle{c}=[circle, draw, solid, fill=black, inner sep=0pt, minimum width=3pt]
\tikzstyle{d}=[diamond, draw, solid, fill=white, inner sep=2pt, minimum width=2pt]
\begin{scope}[thick, every node/.style={sloped,allow upside down}]
	\node[c] at (-3,0){};
	\node[c] at (-2,0){};
	\node[c] at (-1,0){};
	\node[c] at (0,0){};
	\node[c] at (1,0){};
	\node[c] at (2,0){};
	\node[c] at (3,0){};
	\draw[dotted] (-4,0)--(4,0);
	
	\node at (0,1){}; \node at (0,-1){};
\end{scope}
\end{tikzpicture}
\caption{An example of a point set $X$ where $L(X)=\lceil|X|/2\rceil$.}
\label{fig:ex}
\end{figure}

	Finding an asymptotic bound on the layer number of a class of point sets is an intriguing research problem; in particular, discovering a class of point sets where every point set $X$ in the class satisfies $L(X)\leq o(|X|)$ or even $L(X)\leq O(|X|^{1-\epsilon})$ for some $\epsilon>0$ is noteworthy.

	\begin{prob}\label{problem}
		Let $\mathcal{C}$ be a class of point sets in $\mathbb{R}^d$.
		Find the maximum value of $\epsilon>0$ such that $L(X)\leq O(|X|^{1-\epsilon})$ for every finite point set $X\in\mathcal{C}$.
	\end{prob}
	Recently, Problem~\ref{problem} was solved for random point sets~\cite{Dal04} and the square grid~\cite{HPL13}.
	
	In \cite{Dal04}, it was shown that for a bounded region $R$ with non-empty interior in $\mathbb{R}^2$, the expected layer number for an $n$-vertex point set $X$ distributed independently and uniformly in $R$ is $\Theta(n^{2/3})$.
Furthermore, the author of \cite{Dal04} proved the bound is $\Theta(n^{2/(d+1)})$ when $X$ is a set of points in $\mathbb{R}^d$.
	The main result in \cite{HPL13} is that the layer number is $\Theta(n^{2/3})$ for a set of points on the $\cceil{ \sqrt{n} } \times \cceil{ \sqrt{n} }$ square grid.
	It would be remarkable to find the layer number of the $\cceil{ n^{1/d} } \times \cdots \times \cceil{ n^{1/d} }$ $d$-dimensional square grid for each $d\geq 3$; this future direction was also mentioned in~\cite{HPL13}.
	
	Regarding the results in \cite{Dal04} and \cite{HPL13}, it is interesting that the bounds coincide to $\Theta(n^{2/3})$ for the case of an $n$-vertex point set on the plane.
	However, it may not be surprising because the bounds for the number of points in convex position are the same.
	Note that the layer number $L(X)$ of a point set $X$ is bounded below by $|X|$ divided by the maximum number of points in convex position.
	In \cite{Bar89}, it was shown that for an $n$-vertex point set $X$ distributed independently and uniformly in a convex body, the expected value  of $|V(X)|$ is $O(n^{(d-1)/(d+1)})$.
	This result is the core of the proof of the lower bound of $\Omega(n^{2/(d+1)})$ in \cite{Dal04}.
	The expected number of $|V(X)|$ is also $O(n^{(d-1)/(d+1)})$ for grids, as shown in~\cite{Sch88}.
	See \cite{And63, BB91, BL98, Bar08} for an overview and more information on the number of points of the convex layer of sets of lattice points.
Also see \cite{RS63, Sch88, Bar08} for similar results regarding random point sets.

\subsection{Evenly distributed point sets}\label{sec:even}

	Besides sets of random points or sets of lattice points, it is worth considering the layer number of ``evenly distributed'' point sets, that is, point sets that are almost uniformly distributed in a given region.	
In particular, an evenly distributed point set is ``locally'' not too dense.

\begin{definition}
	Let $X$ be a finite point set in a unit ball in $\mathbb{R}^d$.
	For a constant $\alpha > 1$, we say $X$ is \df{$\alpha$-evenly distributed} if	
	\[ |X\cap D|  \leq  \Big\lceil \alpha|X|\Vol(D) \Big\rceil \]
	holds for every Euclidean ball $D$ with positive volume.
\end{definition}

The above definition, roughly speaking, says that when the volume of $D$ is large, then the number of points of $X$ in $D$ is proportional to the volume of $D$, and when the volume of $D$ is small, then the number of points of $X$ in $D$ is bounded above by a constant. 

	In this paper, we investigate Problem~\ref{problem} for $\alpha$-evenly distributed point sets in $\mathbb{R}^d$. 
We succeed in determining the maximum value of $\epsilon$ in Problem~\ref{problem} when $d=2$. 
It was surprising to discover that the behavior of an evenly distributed point set is quite different from that of the square grid.
We also obtain results for higher dimensions, but unfortunately we could not prove that the growth rate is tight. 

	Given a positive integer $d$ and a positive real number $\alpha$, let $\mathcal{C}_d(\alpha)$ be the class of all $\alpha$-evenly distributed point sets in a unit ball in $\mathbb{R}^d$.
	Our main results are the following:
	\begin{theorem}\label{thm1.2}
	For every real number $\alpha > 1$, 
	if $X$ is a point set in the class $\mathcal{C}_2(\alpha)$, then $L(X) \leq O(|X|^{3/4})$, and 
	the growth rate cannot be improved.
	\end{theorem}

	\begin{theorem}\label{thm1.3}
	For every real number $\alpha > 1$ and every integer $d \geq 3$, 
	if $X$ is a point set in the class $\mathcal{C}_d(\alpha)$, then $L(X) \leq O(|X|^{\frac{d+1}{2d}})$.
	\end{theorem}

	In Section~\ref{sec:2dup}, we prove Theorem~\ref{thm1.2} by further developing ideas from \cite{Dal04}.
	In Section~\ref{sec:2dlow}, we give an explicit construction showing that the bound of $O(|X|^{3/4})$ cannot be improved, and hence confirm that $\epsilon=\frac{1}{4}$ is the answer of Problem~\ref{problem} for evenly distributed point sets on the plane.
	In Section~\ref{sec:higher}, we prove an analogue of Theorem~\ref{thm1.2} to higher dimensions, which gives a partial solution of $\epsilon\geq\frac{d-1}{2d}$ to Problem~\ref{problem} for evenly distributed point sets in $\mathbb{R}^d$ where $d\geq 3$.
	We finish the paper with remarks, discussion on future research directions, and open problems in Section~\ref{sec:rmk}.

	\smallskip



\section{Preliminaries}\label{sec:prelim}
We prove some lemmas that will be used in the proofs of the main theorems. 
The statement of the following lemma also appeared in \cite{Dal04}; we include a proof here for completeness. 

\begin{lemma}\label{lem2.1}
	Assume $X$ is a finite point set in $\mathbb{R}^d$.
	\begin{enumerate}[$(a)$]
		\item If $Y\subseteq X$, then $Y\setminus V(Y)\subseteq X\setminus V(X)$.
		\item If $Y\subseteq X$, then $L(Y)\leq L(X)$.
	\end{enumerate}
\end{lemma}
\begin{proof}
$(a)$
If no point of $Y$ is in $V(X)$, then $Y\setminus V(Y)\subseteq Y\subseteq X\setminus V(X)$.
Thus, we may assume that $V(X)\cap Y\neq\emptyset$, and let $v\in V(X)\cap Y$.
Since $v\in V(X)$, there is a hyperplane $H$ such that $H\cap X=\{v\}$ and $H$ divides $\mathbb{R}^d$ into two closed halfspaces $H^+$ and $H^-$ where $H^+\cap X=\{v\}$ and $H^-\cap X=X$.
See Figure~\ref{fig:lem21} for an illustration. 
Since $Y\subseteq X$ and $v\in Y$, it follows that $H^+\cap Y=\{v\}$ and $H^-\cap Y=Y$.
Hence, $v$ is also in $ V(Y)$.
		This shows that $Y\cap V(X)\subseteq V(Y)$.
		Therefore, $Y\setminus V(Y)\subseteq Y\setminus (Y\cap V(X))=Y\setminus V(X) \subseteq X\setminus V(X)$.

\begin{figure}[h]
    \centerline{\includegraphics[scale=0.8]{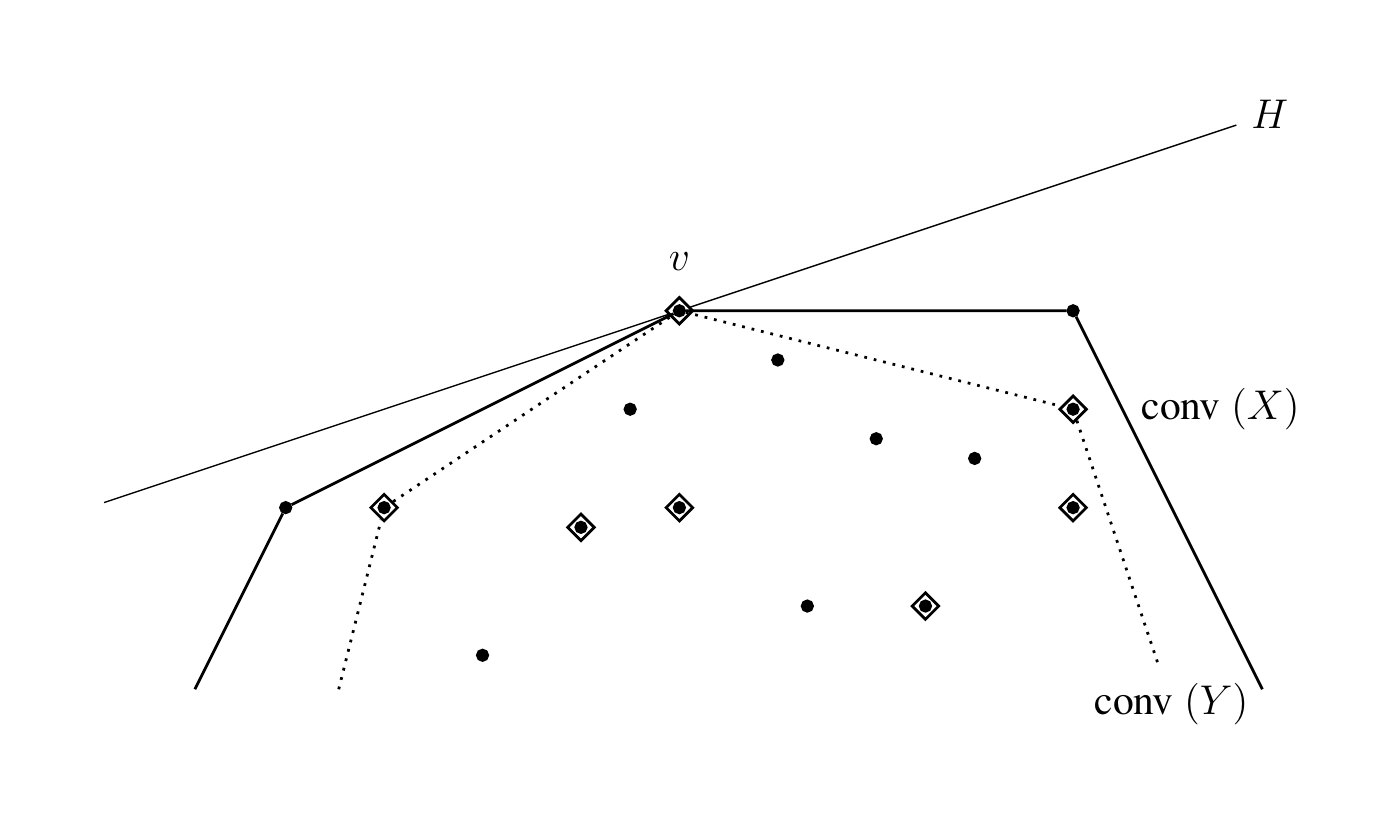}}
    \caption{For Lemma~\ref{lem2.1}~$(a)$. Dots represent points in $X$, and diamonds represent points in $Y$. Note that $Y\subseteq X$.} \label{fig:lem21}
\end{figure}

		$(b)$ 
Let $\{X_i\}_i$ and $\{Y_i\}_i$ be the peeling sequence of $X$ and $Y$, respectively.
		We will show that $Y_k\subseteq X_k$ for every $k\geq1$ by induction on $k$.
		For $k=1$, it is true since $Y\subseteq X$.
		Assume $k\geq2$ and $Y_{k-1}\subseteq X_{k-1}$ holds.
		By $(a)$, we know $Y_{k-1}\setminus V(Y_{k-1})\subseteq X_{k-1}\setminus V(X_{k-1})$, which is equivalent to $Y_{k}\subseteq X_{k}$.
Hence, when $k=L(X)+1$, it follows that $Y_{L(X)+1}\subseteq X_{L(X)+1}=\emptyset$, which means $L(Y)\leq L(X)$.
	\end{proof}

	Let $K_1$ and $K_2$ be two convex bodies in $\mathbb{R}^d$ such that $K_2 \subset \intr(K_1)$.
	For each point $p$ on the boundary of $K_2$, let $H_p$ be a tangent hyperplane of $K_2$ at the point $p$.
	Note that $H_p$ divides $\mathbb{R}^d$ into two closed halfspaces, say $H^+_p$ and $H^-_p$ such that $\intr(H^+_p) \cap K_2 = \emptyset$.
	The closed region $K_1 \cap H^+_p$ is a called a \df{cap} of $K_1\setminus K_2$.
	See Figure~\ref{fig:cap} for an illustration.
	Note that each cap has a corresponding hyperplane and a corresponding point on the boundary of $K_2$. 
	The following lemma shows how the number of points in a cap is related to the layer number of the given point set in $\mathbb{R}^d$.

\begin{figure}[h]
    \centerline{\includegraphics[scale=1]{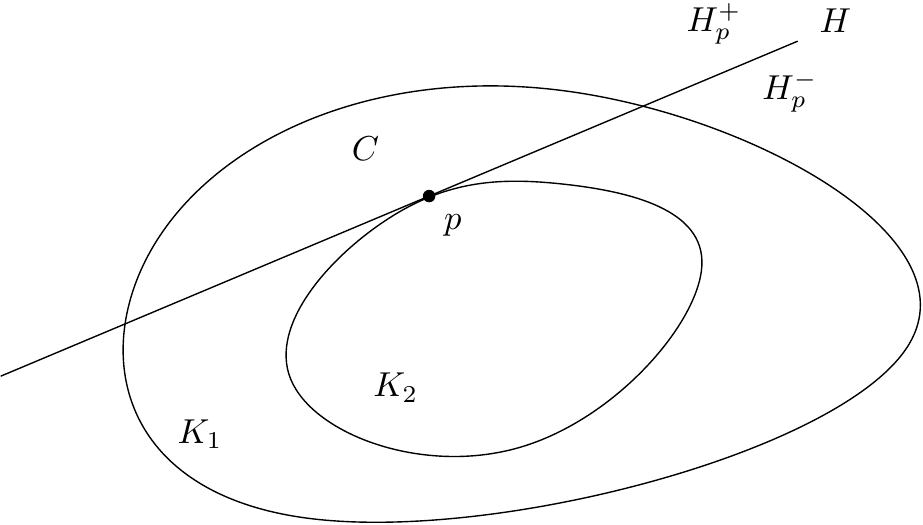}}
    \caption{A cap $C$ of $K_1 \setminus K_2$ } \label{fig:cap}
\end{figure}

\begin{lemma}\label{lem2.3}
	If $K_1$ and $K_2$ are two convex bodies in $\mathbb{R}^d$ such that $K_2\subseteq \intr(K_1)$ and $X$ is a finite point set in $K_1$, then
	$L(X) \leq \max\{|X\cap C|: C\text{ is a cap of }K_1\setminus K_2\} +L(X\cap K_2)$.
\end{lemma}

\begin{proof}
Let $C$ be a cap of $K_1\setminus K_2$ and let $H$ be the corresponding hyperplane of $C$. 
	Recall that $H$ divides $\mathbb{R}^d$ into two closed halfspaces $H^+$ and $H^-$ where $C=K_1\cap H^+$.

Let $\{X_i\}_i$ be the peeling sequence of $X$.
We claim that for $k\geq 1$, either $V(X_k)\cap C$ is non-empty or $X_k\cap C$ is empty.
	Suppose $X_k\cap C$ is non-empty.
	Note that $X_k\cap C = X_k\cap H^+$ since $X\subset K_1$.
	Now, it must be that $\conv(X_k)$ contains some point $p$ in $H^+$ such that $p\in V(X_k)$.
	Since $p$ must also be in $C$, it follows that $V(X_k)\cap C$ is non-empty.
	
	Now, at every peeling step, at least one point is removed from each non-empty cap.
	Let $C_0$ be a cap of $K_1 \setminus K_2$ that contains the maximum number of points of $X$.
	By repetitively applying Lemma~\ref{lem2.1} $(b)$, we conclude that $L(X) \leq |X\cap C_0|+L(X_{|X\cap C_0|+1}\cap K_2)\leq |X\cap C_0|+L(X\cap K_2)$.
\end{proof}

	Recall that a finite point set $X$ in a unit ball in $\mathbb{R}^d$  is $\alpha$-evenly distributed if there is a constant $\alpha > 1$ such that	
	\[ |X\cap D|  \leq  \lceil \alpha|X|\Vol(D) \rceil \]
	holds for every Euclidean ball $D$ with positive volume.		
	Let $\delta(X)$ be the minimum distance between two different points in $X$; namely, $\delta(X) = \min \{ \|x-y\| : x,y\in X\text{ and } x \neq y\}$.	
	The following lemma gives a sufficient condition on $\delta(X)$ in order for $X$ to be an $\alpha$-evenly distributed point set for some $\alpha>1$.

	\begin{lemma}\label{lem2.4}
	For a positive integer $d$, there exists a continuous bijection $f_d:\mathbb{R}_{>0}\to\mathbb{R}_{>1}$ such that 
	if $X$ is a finite point set in a unit ball in $\mathbb{R}^d$ satisfying  $\delta(X) \geq \beta{|X|^{-1/d}}$ for some constant $\beta>0$, 
	then $X$ is $f_d(\beta)$-evenly distributed.
\end{lemma}

\begin{proof}
	Let $n=|X|$ and let $\beta>0$ be a constant such that $\delta(X)\geq \beta{n^{-1/d}}$.
	Also, let $B$ be a ball of radius $r>0$ in $\mathbb{R}^d$.
	Note that the $\Vol(B)=C_d r^d$ for some positive constant $C_d$.
	A $d$-dimensional cube with side length $\frac{\delta(X)}{2\sqrt{d}}$ contains at most one point of $X$ since it fits in a ball of diameter less than $\delta(X)$. 
	Since $B$ fits in a $d$-dimensional cube with side length $2r$, 
	we have the following upper bound on $|X\cap B|$:
	\[|X\cap B|\leq \left(\left\lfloor{\frac{2r}{\delta(X)/\sqrt{4d}}}\right\rfloor + 1\right)^d\leq \left\lceil\left(\frac{4\sqrt{d}r}{\beta} n^{1/d}+1\right)^d\right\rceil.\]
	Let $f_d(\beta) = \left(\frac{4\sqrt{d}}{C_d^{1/d}\beta } + 1\right)^d$ so $f_d$ is a continuous bijection from $\mathbb{R}_{>0}$ to $\mathbb{R}_{>1}$. 
	We see that $X$ is $f_d(\beta)$-evenly distributed by definition.
\end{proof}



\section{The layer number of an evenly distributed point set on the plane}\label{sec:planar}

\subsection{The upper bound on the plane}\label{sec:2dup}
	In this section, we give an upper bound of $O(|X|^{3/4})$ on the layer number of a point set $X$ in the class $\mathcal{C}_2(\alpha)$, that is, an $\alpha$-evenly distributed point set in a unit disk in $\mathbb{R}^2$.
	This shows the first part of Theorem~\ref{thm1.2}.

\begin{prop}\label{prop:2dup}
For a real number $\alpha > 1$, 
	if $X$ is a point set in the class $\mathcal{C}_2(\alpha)$, then $L(X) \leq O(|X|^{3/4})$.
\end{prop}
\begin{proof}
	Let $X \in \mathcal{C}_2(\alpha)$ with sufficiently many points where $n = |X|$ and let $N = \lfloor \sqrt{n} \rfloor$.
	For each $j\in\{0, \ldots, N\}$, let $D_j$ be the disk in $\mathbb{R}^2$ with radius $1-\frac{j}{\sqrt{n}}$ centered at the origin;
	note that $D_0$ is the unit disk  centered at the origin.
	For each $j\in\{0, \ldots, N-1\}$, let $C_j$ be a cap of $D_j\setminus D_{j+1}$ that contains the maximum number of points of $X$ among all caps of $D_j\setminus D_{j+1}$.
See Figure~\ref{fig:thm311}.

\begin{figure}[h]
    \centerline{\includegraphics[scale=0.8]{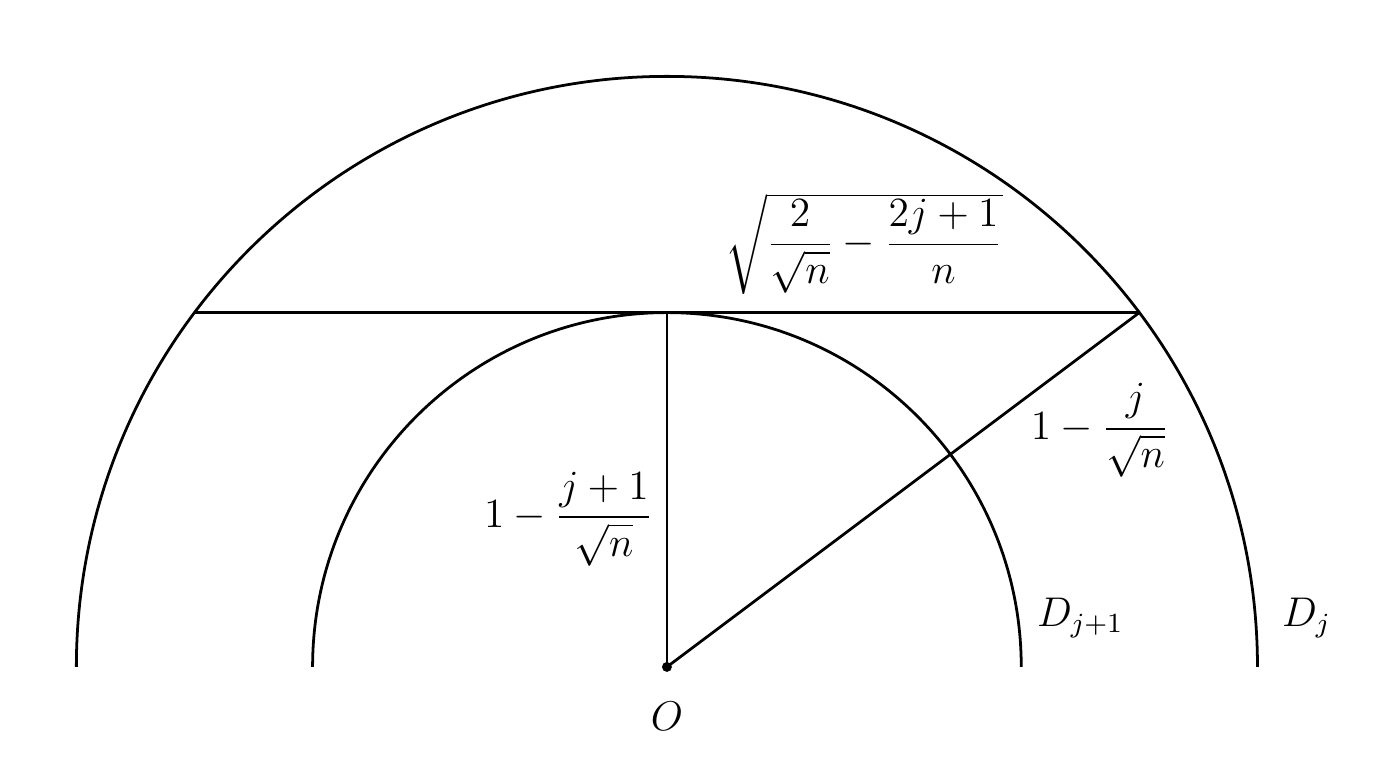}}
    \caption{Take a cap $C_j$ of $D_j \setminus D_{j+1}$ that contains the maximum number of points of $X$.} 
\label{fig:thm311}
\end{figure}
			
	Since $X$ is $\alpha$-evenly distributed in $D_0$, we have upper bounds 
	\[|X \cap B_0| \leq \lceil\alpha\pi\rceil \;\;\;\text{and}\;\;\; |X \cap B_1| \leq \left\lceil\frac{\alpha\pi}{2}\right\rceil,\]
	where $B_0$ and $B_1$ is a disk of radius ${{n^{-1/2}}}$ and $\frac{1}{\sqrt{2}}n^{-1/2}$, respectively.
	We prove the following two lemmas, which provide upper bounds on $|X\cap C_j|$ and 
	$|X\cap D_{N}|$.
		
	\begin{lemma}\label{lem3.1}
	For $j\in\{0, \ldots, N-1\}$, we know $|X\cap C_j|\leq O(n^{1/4}) + O(1)$.
	\end{lemma}

	\begin{proof}
	For each $j\in\left\{0, \ldots, N-1\right\}$, the area of $C_j$ is upper bounded by the area of the smallest rectangle $R_j$ containing $C_j$.
	The length of one side of $R_j$ is $n^{-1/2}$ and the length of the other side of $R_j$ is bounded above by $2\sqrt{2}n^{-1/4}$ since $2\sqrt{2n^{-1/2} - (2j+1)/n}\leq 2\sqrt{2}n^{-1/4}$.
	By partitioning $R_j$ into $\cceil{2\sqrt{2}n^{1/4}}$ squares of side length $n^{-1/2}$, we see that $R_j$ can be covered by $\cceil{2\sqrt{2}n^{1/4}}$ disks of radii $\frac{1}{\sqrt{2}}n^{-1/2}$.
See Figure~\ref{fig:lem312}.

\begin{figure}[h]
    \centerline{\includegraphics[scale=0.75]{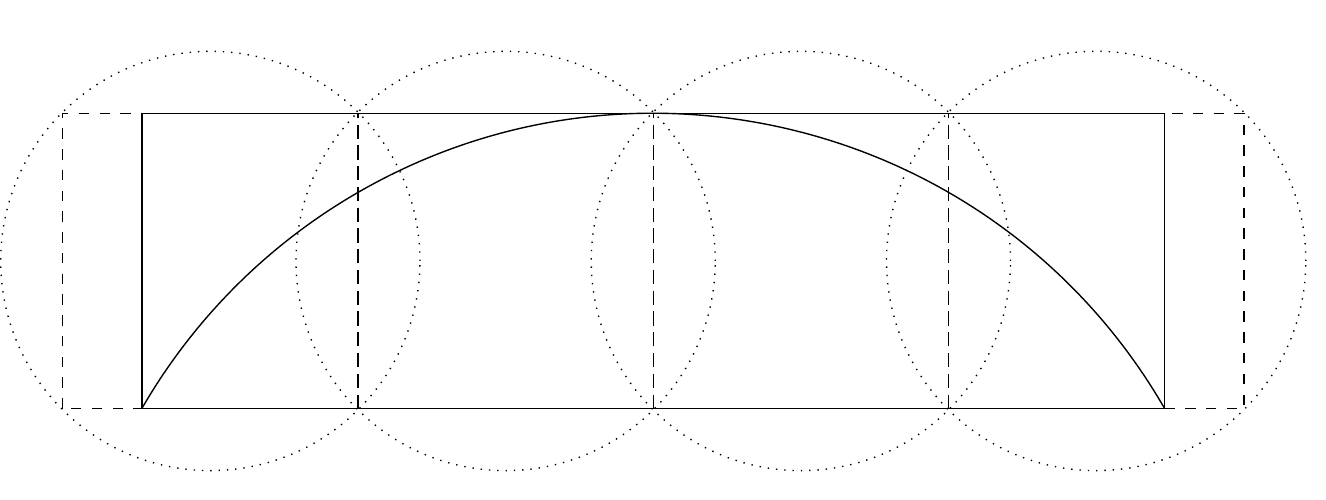}}
    \caption{A rectangle $R_j$ containing a cap $C_j$ can be covered by disks of radii $\frac{1}{\sqrt{2}}n^{-1/2}$.} \label{fig:lem312}
\end{figure}	
			
	Since each disk of radius $\frac{1}{\sqrt{2}}n^{-1/2}$ contains at most $\cceil{\frac{\alpha\pi}{2}}$ points of $X$, 
	it follows that 
	\begin{align*}
	|X\cap C_j| \quad &\leq \quad |X\cap R_j|\\
	&\leq \quad \cceil{\frac{\alpha\pi}{2}} \cceil{2\sqrt{2}n^{1/4}}\\
	&\leq\quad  \cceil{\frac{\alpha\pi}{2}} (2\sqrt{2}n^{1/4}+1)\\
	&=\quad  O(n^{1/4})+O(1).
	\end{align*}
	\end{proof}

	\begin{lemma}\label{lem3.2}
		$|X\cap D_N|\leq \lceil \alpha \pi\rceil$.
	\end{lemma}
	\begin{proof}
	This is obvious since $D_N$ is a disk of radius less than $n^{-1/2}$. 
	\end{proof}

Now, 
\begin{align*}
L(X)	\quad 
&\leq \quad |X\cap D_N|+ \sum_{j=0}^{ N-1} |X\cap C_j|\tag{by Lemma~\ref{lem2.3}}\\
	&\leq\quad  \lceil\alpha\pi\rceil+N\times(O(n^{1/4})+O(1)) \tag{by Lemmas~\ref{lem3.1} and~\ref{lem3.2}}\\
	&\leq \quad \lceil\alpha\pi \rceil + n^{1/2}(O(n^{1/4})+O(1))\\
	&\leq \quad O(n^{3/4})
\end{align*}

Hence, $L(X) \leq O(|X|^{3/4})$.
		\end{proof}

\subsection{The lower bound on the plane}\label{sec:2dlow}

In this section, we prove that the upper bound in Theorem~\ref{thm1.2} is best possible.

\begin{prop}\label{thm:2dlow}
	For every real number $\alpha > 1$, there exists an $\alpha$-evenly distributed point set $X$ in a unit disk in $\mathbb{R}^2$ such that $L(X) \geq \Omega(|X|^{3/4})$.
\end{prop}
\begin{proof}
	We give an explicit construction of an $\alpha$-evenly distributed point set $X$ in the unit disk $D$ in $\mathbb{R}^2$ centered at the origin with $|X| \leq O(n)$ and $L(X)\geq\Omega(n^{3/4})$.
	Fix $\alpha > 1$, and let $\beta> 0$ satisfy $f_2(\beta) = \alpha$, where $f_2$ is the function defined in Lemma~\ref{lem2.4}.
	Let $C = \max\{2\beta, 4\pi^2\}$ and $M = \ffloor{ \frac{\sqrt{n}}{C} }$.
	For a sufficiently large positive integer $n$, let $k = \ffloor{ n^{1/4} }$.

	Let $P_1$ be a regular $k$-gon centered at the origin such that the distance between the origin and a vertex of $P_1$ is $\frac{C}{\sqrt{n}}$.
	For each $j\in\{1, \ldots, M\}$, let $P_j$ be the scaled copy of $P_1$, centered at the origin, with scaling factor $j$.
	Let $l_j$ be the length of an edge of $P_j$.
	Since $\frac{2x}{\pi} \leq \sin{x} < x$ and $k > n^{1/4}/2$, we have
	\[\frac{j}{n^{3/4}} \leq l_j = 2 \times \frac{C}{\sqrt{n}}j \times \sin{\frac{\pi}{k}} < \frac{4C\pi j}{n^{3/4}}.\]

	Now, for each $j\in\{1, \ldots, M\}$, define a point set $Q_j$ containing all vertices of $P_j$, and in addition include as many points as possible on the boundary of $P_j$ such that the distance between two (consecutive) points on $Q_j$ is at least $\beta n^{-1/2}$.
	Finally, let \[X = \bigcup_{j=1}^{M} Q_j.\]
	See Figure~\ref{fig:thm321} for an illustration.
\begin{figure}[h]
    \centerline{\includegraphics[scale=0.25]{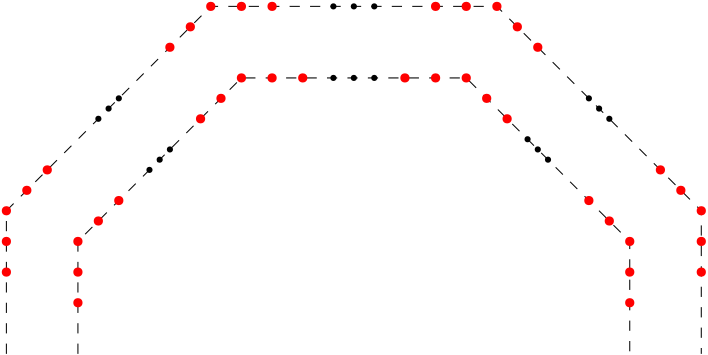}}
    \caption{An $\alpha$-evenly distributed point set $X$ in a unit disk in $\mathbb{R}^2$ such that $L(X) = \Omega(|X|^{3/4})$.} \label{fig:thm321}
\end{figure}

	Since $\cos\frac{\pi}{k} \geq \frac{1}{2}$ for sufficiently large $k$, the minimum distance between points on the boundaries of $P_j$ and $P_{j+1}$ is at least 
	\[\frac{C}{\sqrt{n}} \times \cos\frac{\pi}{k} \geq \frac{C}{2\sqrt{n}} \geq \beta n^{-1/2}.\]	
	Hence $\delta(X)\geq \beta n^{-1/2}$, and by Lemma~\ref{lem2.4}, $X$ is $\alpha$-evenly distributed in $D$.
	
\begin{claim}\label{claim3.4}
	For $j\in\{2, \ldots, M\}$, if $P'_j$ is  a regular $k$-gon whose vertices are the midpoints of edges of $P_j$, then $P_{j-1} \subset P'_j$.
\end{claim}		
\begin{proof}
	Let $O$ be the origin and let $W_1$, $V$, and $W_2$ be three consecutive vertices of $P_k$. 
For $i\in\{1, 2\}$, let $M_i$ be the midpoint of the edge ${VW_i}$. 
Let $U$ be the vertex of $P_{k-1}$ such that $O,U,V$ are collinear, and let $P$ be the point where $UV$ and $M_1 M_2$ intersect.
	It is sufficient to show that $P$ is strictly between $U$ and $V$. 
See Figure~\ref{fig:claim322} for an illustration. 

\begin{figure}[h]
    \centerline{\includegraphics[scale=0.25]{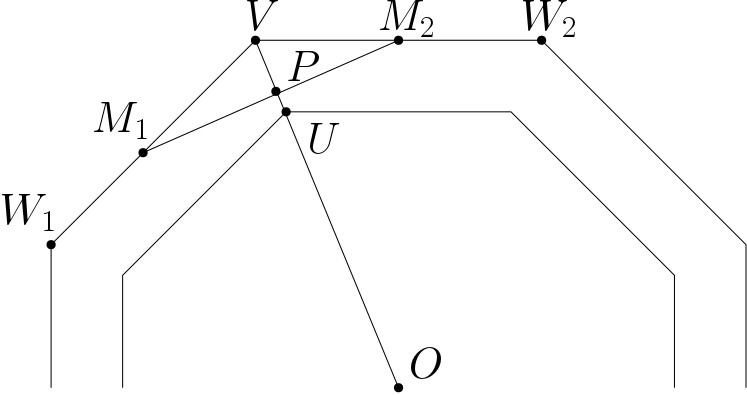}}
    \caption{An illustration of Claim~\ref{claim3.4}. $P$ exists strictly between $U$ and $V$.} \label{fig:claim322}
\end{figure}	

	Let $l\leq1$ be the length of the line segment $OV$.
	Since $\angle W_i O V = 2\pi/k$ and $W_iOV$ is an isosceles triangle for each $i\in\{1, 2\}$, we know that $\angle M_1 O V = \angle V M_1 M_2 = \pi/k$.
	Since $\sin x < x $ for $x > 0$ and $k > n^{1/4}/2$, the length of the line segment $VP$ is given by 
	\[l \times \left(\sin\frac{\pi}{k}\right)^2< 1\times \left(\frac{\pi}{k}\right)^2 = \frac{\pi^2}{k^2} < \frac{4\pi^2}{\sqrt{n}} \leq \frac{C}{\sqrt{n}}.
	\]
	Since the line segment $VP$ has length strictly less than the length of $UV$, which is $Cn^{-1/2}$, it follows that $P$ is strictly between $U$ and $V$.
\end{proof}

\begin{claim}\label{claim3.5}
	For $j\in\{1, \ldots, M\}$, we have $|Q_j| = \Theta(j)$ and $L(Q_j) \geq \frac{j}{4\beta n^{1/4}}$.
\end{claim}

\begin{proof}
	Recall that 
	\[\frac{j}{n^{3/4}} \leq l_j < \frac{4C\pi j}{n^{3/4}}.\]
	Since the distance between two consecutive points of $Q_j$ is between $\beta n^{-1/2}$ and $2\beta n^{-1/2}$, we obtain
	\[\frac{1}{2\beta}j \leq \frac{l_j}{2\beta n^{-1/2}}\times n^{1/4} \leq |Q_j| \leq \frac{l_j}{\beta n^{-1/2}}\times n^{1/4} \leq \frac{4C\pi}{\beta}j.\]
	Thus, we conclude that $|Q_j| = \Theta(j)$.	

	For the layer number of $Q_j$, observe that every layer of $Q_j$ contains at most two points on each edge of $P_j$.
In fact, the layer number of $Q_j$ is at least half the number of points on an edge. 
Therefore, 
	\[L(Q_j) \geq \frac{l_j/2}{2 \beta n^{-1/2}} \geq \frac{j/n^{3/4}}{4\beta n^{-1/2}} = \frac{j}{4\beta n^{1/4}}.\]
\end{proof}

\begin{claim}\label{claim3.6}
	$L(X)\geq \Omega(n^{3/4})$ and $|X|= \Theta(n)$.
\end{claim}
\begin{proof}
	Note that 
	\[\sum_{j=1}^{M} j = \frac{M(M+1)}{2} = \Theta(n).\]
	Thus $|X| = \sum_{j=1}^{M} |Q_j| = \Theta(n)$.

	On the other hand, Claim~\ref{claim3.4} tells us that all points of $Q_k$ will be removed before removing a point of $Q_{k-1}$ in the peeling process of $X$. 
	By Claim~\ref{claim3.5}, we obtain the following lower bound on $L(X)$.
	\[
	L(X) \geq \sum_{j=1}^{M} \frac{j}{4\beta n^{1/4}} = \Omega(n^{3/4}).
	\]
		\end{proof}
	Since $L(X) \geq \Omega(n^{3/4})$ and $|X| =\Theta(n)$, we know that $L(X)\geq\Omega(|X|^{3/4})$.
\end{proof}



\section{Results in higher dimensions}\label{sec:higher}

	In this section, we generalize Proposition~\ref{prop:2dup} in Section~\ref{sec:2dup} to higher dimensions by proving Theorem~\ref{thm1.3}.
	The proof is similar to the planar case, yet, it is not a straightforward generalization. 

\begin{proof}[Proof of Theorem~\ref{thm1.3}]

	Fix $d\geq 3$ and let $D_0$ be a unit ball in $\mathbb{R}^d$ centered at the origin.
	Let $X$ be an $\alpha$-evenly distributed finite point set in $D_0$ with sufficiently many points and let $n=|X|$.
		The volume of a $d$-dimensional ball with radius $r>0$ is given by $\dfrac{\pi^{d/2}}{\Gamma\left(\frac{d}{2}+1\right)}r^d = \Theta(r^d)$ for a fixed $d$;
	$\Gamma$ is Euler's gamma function. 
	Hence, for a $d$-dimensional ball to have at most a constant number of points of $X$, the ball must have radius $O(n^{-1/d})$.
	Fix $k \geq 0$ and let $B_0$, $B_1$, and $B_2$ be a ball of radius $n^{-k}$, $\frac{\sqrt{d}}{2}n^{-1/d}$, and $\frac{\sqrt{d}}{2}n^{-k}$, respectively.
	Since $X$ is $\alpha$-evenly distributed in $D_0$, there exist constants $c_1, c_2, c_3, c_4 > 0$ depending only on $\alpha, d$, and $k$, but not $n$, such that 
	\[|X \cap B_0| \leq \begin{cases} c_1 & \text{ if }k > \frac{1}{d} \\ c_2 n^{1-dk} & \text{ if }0 \leq k \leq \frac{1}{d} \end{cases}, |X \cap B_1| \leq c_3 , \text{ and } |X \cap B_2| \leq c_4 n^{1-dk}.\]

	Let $N = \ffloor{ n^k }$.
	For each $j\in\{0, \ldots, N\}$, let $D_j$ be the ball of radius $1 - \frac{j}{n^k}$ centered at the origin.
	For each $j\in\{0, \ldots, N-1\}$, let $C_j$ be a cap of $D_j\setminus D_{j+1}$ that contains the maximum number of points of $X$ among all caps of $D_j\setminus D_{j+1}$.
	
	The following lemmas reveal upper bounds on $|X\cap C_j|$ and $|X\cap D_{N}|$.

	\begin{lemma}\label{lem4.1}
	For $j\in\{0, \ldots, N-1\}$, we know 
	\[|X\cap C_j|\leq \begin{cases} O(n^{\frac{1}{2}-\frac{1}{2d}}) & \text{ if }k > \frac{1}{d} \\ O(n^{1-\frac{dk}{2}-\frac{k}{2}}) &\text{ if }0 \leq k \leq \frac{1}{d} \end{cases}.\]
	\end{lemma}
	\begin{proof}
	For each $j\in\{0,\ldots, N-1\}$, $C_j$ is bounded by the smallest $d$-dimensional box $R_j$ containing $C_j$.
	Similar to the proof of Lemma~\ref{lem3.1}, 
	by geometric observations we know that $R_j$ has $d-1$ sides of length $2\sqrt{2n^{-k}-(2j+1)n^{-2k}} \leq 2\sqrt{2}n^{-k/2}$ and one side of length $n^{-k}$.
	
	If $k \geq 1/d$, then by partitioning $R_j$ into at most $\cceil{ 2\sqrt{2}n^{\frac{1}{d}-\frac{k}{2}} }^{d-1}$ $d$-dimensional boxes of side length $n^{-1/d}$, we see that $R_j$ can be covered by $4^{d-1}n^{\frac{d-1}{2d}}$ balls of radii $\frac{\sqrt{d}}{2}n^{-1/d}$.
	If $0 \leq k \leq 1/d$, then by partitioning $R_j$ into at most $\cceil{ 2\sqrt{2}n^{k/2} }^{d-1}$ $d$-dimensional boxes of side length $n^{-k}$, we see that $R_j$ can be covered by $4^{d-1} n^{(d-1)k/2}$ balls of radii $\frac{\sqrt{d}}{2} n^{-k}$.
	
	It follows that 
	\begin{align*}
	|X\cap C_j| \quad &\leq\quad  |X\cap R_j|\\
	&\leq \begin{cases} c_3 4^{d-1} n^{\frac{d-1}{2d}}  & \text{ if }k \geq \frac{1}{d} \\ 
	c_4 n^{1-dk} \times 4^{d-1} n^{(d-1)k/2} &\text{ if }0 \leq k \leq \frac{1}{d} \end{cases}\\
	&\leq \begin{cases} O(n^{\frac{1}{2}-\frac{1}{2d}}) & \text{ if }k \geq \frac{1}{d} \\ 
	O(n^{1-\frac{dk}{2}-\frac{k}{2}}) &\text{ if }0 \leq k \leq \frac{1}{d} \end{cases}.
	\end{align*}
	\end{proof}

	\begin{lemma}\label{lem4.2}
		$|X\cap D_N|\leq \begin{cases} O(1) & \text{ if }k \geq \frac{1}{d} \\ O(n^{1-dk}) &\text{ if }0 \leq k \leq \frac{1}{d} \end{cases}$.
	\end{lemma}
	\begin{proof}
	This is obvious since $D_N$ is a ball of radius less than $n^{-k}$.
	\end{proof}
	
		Recall that $X$ is an $\alpha$-evenly distributed finite point set in a unit ball in $\mathbb{R}^d$. Now,
	\begin{align*}
		L(X)\quad 
	&\leq\quad |X\cap D_N|+ \sum_{j=0}^{N-1} |X\cap C_j|\tag{by Lemma~\ref{lem2.3}}\\
	&\leq\quad  \begin{cases} (\sum_{j=0}^{N-1} O(n^{\frac{1}{2} - \frac{1}{2d}})) + O(1) & \text{ if }k \geq \frac{1}{d} \\ (\sum_{j=0}^{N-1} O(n^{1 - \frac{dk}{2} - \frac{k}{2}})) + O(n^{1-dk}) &\text{ if }0 \leq k \leq \frac{1}{d} \end{cases}\tag{by Lemmas~\ref{lem4.1} and~\ref{lem4.2}}\\
	&\leq\quad  \begin{cases} O(n^{\frac{1}{2} - \frac{1}{2d} + k}) & \text{ if }k \geq \frac{1}{d} \\  O(n^{1 - \frac{dk}{2} + \frac{k}{2}}) &\text{ if }0 \leq k \leq \frac{1}{d} \end{cases}\tag{since $N \leq n^k$}\\
	&\leq\quad  O(n^{\frac{1}{2} + \frac{1}{2d}})\tag{let $k={1/d}$}		
	\end{align*}
Hence, $L(X) \leq O(|X|^{\frac{d+1}{2d}})$.
		\end{proof}

\section{Future directions and open problems}\label{sec:rmk}

	In this section, we suggest possible future research directions and open problems.

	In the planar case, we gave a construction to show tightness of the upper bound of Theorem~\ref{thm1.2}.
	For higher dimensions, it is trickier to construct such examples relying only on Euclidean geometry.
	One possible way to find such an example is to consider vector configurations.
	We conjecture that the upper bound we found for higher dimensions is also tight. 
	
\begin{conj}\label{conj1}
	For every real number $\alpha>1$ and $d \geq 3$, there exists an $\alpha$-evenly distributed point set $X$ in $\mathbb{R}^d$ such that $L(X) \geq \Omega(|X|^{\frac{d+1}{2d}})$.
\end{conj}

	In this paper, we focused on upper bounds of the layer number. 
	However, one can also ask about lower bounds.

\begin{que}\label{que1}
	For a real number $\alpha>1$ and $d\geq 2$, find a lower bound on $L(X)$ for an $\alpha$-evenly distributed finite point set $X$ in the unit ball in $\mathbb{R}^d$.
\end{que}
	For $d=2$, one can easily construct an example of an $\alpha$-evenly distributed point set $X$ on the plane such that $L(X)\geq\Omega({|X|}^{1/2})$.
	We suspect that $\Omega({|X|}^{1/d})$ is also the best possible lower bound for  $d\geq 2$.

	Finally, recall that in \cite{HPL13}, it was shown that for the $\cceil {n^{1/2} }\times \cceil {n^{1/2}} $ square grid $G_n$ in $\mathbb{R}^2$, $L(G_n) = \Theta(n^{2/3})$.
	However, it is quite challenging to generalize their proof for higher dimensions.
	The authors of \cite{HPL13} also posed the following question. 

\begin{que}[\cite{HPL13}]\label{que2}
	If $G_n ^d$ is the $\cceil {n^{1/d}}\times \cdots \times \cceil {n^{1/d} }$ square grid in $\mathbb{R}^d$ for $d \geq 3$, then what can one say about $L(G_n ^d)$?
\end{que}

	Note that the point set $G_n^d$ is an $\alpha$-evenly distributed point set in a unit ball in $\mathbb{R}^d$ for some $\alpha >1$ by applying Lemma~\ref{lem2.4} after rescaling $G_n^d$.
	By Theorem~\ref{thm1.2}, we know $L(G^d_n) \leq O(n^{2/d})$ for $d\geq 3$.
	On the other hand, the lower bound on $L(G^d_n)$ can be obtained as an immediate consequence of the following well-known result (see \cite[Theorem 13.1]{Bar08}, for example).
	\begin{theorem}[\cite{Bar08}]
	If $X \subset \mathbb{Z}^d$, then 	$|V(\conv(X))| << (\Vol(\conv(X)))^{(d-1)/(d+1)}$.
	\end{theorem}
	From the above theorem, we obtain $L(G^d_n) \geq \Omega(n^{2/(d+1)})$ since the number of points deleted in each step of the peeling process of $G^d_n$ is $O(n^{(d-1)/(d+1)})$.

	Indeed, it was shown in \cite{BB91} that the number of lattice points on a ball of radius $r>0$ in $\mathbb{R}^d$ is $\Theta(r^{(d-1)/(d+1)})$.
	Hence, we suspect that $L(G_n^d)=\Theta(n^{2/(d+1)})$, which is same as the expected layer number of a random point set $X$ with $n$ points in a unit ball of $\mathbb{R}^d$~\cite{Dal04}.

\section*{Acknowledgments}
The authors thank Andreas Holmsen for helpful comments and development of this project in the early stages, and also Joseph Briggs for improving the readability of the manuscript.
We thank the referees for pointing out the various mistakes and providing valuable comments, which improved the quality of the manuscript.

\end{document}